\documentclass[reqno]{amsart}
\usepackage{enumerate}
\usepackage{amsfonts,amssymb,amsmath,amsthm}
\usepackage{epsfig}
\usepackage{graphics}
\usepackage[normalem]{ulem}
\usepackage{color}
\usepackage{version}

\input xy % these two lines are for commutative diagrams
\xyoption{all}
\numberwithin{equation}{section}

\definecolor{OrangeRed}{cmyk}{0,0.6,1,0}            % half magenta only, full yellow
\definecolor{DarkBlue}{cmyk}{1,1,0,0.20}
\definecolor{DarkGreen}{cmyk}{1,0,0.6,0.2}
\definecolor{myblue}{rgb}{0.66,0.78,1.00}
\definecolor{Violet}{cmyk}{0.79,0.88,0,0}
\definecolor{Lavender}{cmyk}{0,0.48,0,0}

\newtheorem{thm}{Theorem}[section]
\newtheorem{theorem}[thm]{Theorem}
\newtheorem{main theorem}[thm]{Main Theorem}
\newtheorem{corollary}[thm]{Corollary}

\newtheorem{lemma}[thm]{Lemma}
\newtheorem{lem}[thm]{Lemma}
\newtheorem{prop}[thm]{Proposition}

\theoremstyle{definition}

\newtheorem{defn}[thm]{Definition}
\newtheorem{definition}[thm]{Definition}
\newtheorem{remark}[thm]{Remark}

\newtheorem{example}[thm]{Example}

\def\C{\mathbb C}

\def\bcases{\begin{cases}}

\def\ecases{\end{cases}}

\newcommand{\D}{\mathbb D}

\newcommand{\N}{\mathbb N}

\newcommand{\R}{\mathbb R}

\newcommand{\FF}{\mathcal F}
\newcommand{\bea}{\begin{eqnarray*}}
\newcommand{\eea}{\end{eqnarray*}}

\newcommand{\be}{\begin{equation}}
\newcommand{\ee}{\end{equation}}

\newcommand{\ra}{\rightarrow}

\renewcommand{\epsilon}{\varepsilon}
\renewcommand{\phi}{\varphi}

\newcommand{\ov}{\overline}
\newcommand{\nk}{{n_k}}

\newcommand{\htop}{h_{\operatorname{top}}}

\begin{document}

\title{Dynamics of transcendental H\'enon maps III: Infinite entropy}

\author[L. Arosio]{Leandro Arosio$^{\dag}$}
\author[A.M. Benini]{Anna Miriam Benini$^{\ddag}$}
\author[J.E.  Forn{\ae}ss ]{John Erik Forn{\ae}ss}
\author[H. Peters]{Han Peters}

\today
\thanks{$^{\dag}$  Supported by the SIR grant ``NEWHOLITE - New methods in holomorphic iteration'' no. RBSI14CFME. Partially supported by the MIUR Excellence Department Project awarded to the Department of Mathematics, University of Rome Tor Vergata,  CUP E83C18000100006.}
\thanks{$^{\ddag}$ This project has been partially supported by the project 'Transcendental Dynamics 1.5' inside the program FIL-Quota Incentivante of the  University of Parma and co-sponsored by Fondazione Cariparma, and by Indam via the research group GNAMPA}
\address{L. Arosio: Dipartimento Di Matematica\\
Universit\`{a} di Roma \textquotedblleft Tor Vergata\textquotedblright\  \\
 Italy} \email{arosio@mat.uniroma2.it}
\address{A.M. Benini: Dipartimento  di   Matematica Fisica e Informatica\\
Universit\'a di Parma, IT.  \\
} \email{ambenini@gmail.com}
\address{H. Peters: Korteweg de Vries Institute for Mathematics\\
University of Amsterdam\\
the Netherlands} \email{hanpeters77@gmail.com}
\address{ J.E. Fornaess: Department of Mathematical Sciences\\
NTNU Trondheim, Norway} \email{john.fornass@ntnu.no}

\begin{abstract}
Very little is currently known about the dynamics of non-polynomial entire maps in several complex variables. The family of transcendental H\'enon maps offers the potential of combining ideas from transcendental dynamics in one variable, and the dynamics of polynomial H\'enon maps in two. Here we show that these maps all have infinite topological and measure theoretic entropy. The proof also implies the existence of infinitely many periodic orbits of any order greater than two.
\end{abstract}

\maketitle

\section{Introduction}

A \emph{transcendental H\'enon map} is a holomorphic automorphism of $\mathbb C^2$ of the form
$$
F(z,w) = (f(z) - \delta w, z),
$$
where $\delta \in \mathbb C \setminus \{0\}$, and $f$ is a transcendental entire function. Transcendental H\'enon maps form a bridge between two distinct families of holomorphic maps whose dynamical behaviors have been studied intensively in recent years: the family of complex (polynomial) H\'enon maps, and the family of transcendental entire functions.

In two previous papers \cite{henon1, henon2} we studied the dynamics of these maps, demonstrating non-trivial dynamical behavior. For example, the Julia set is always non-empty. Here we provide further evidence of non-trivial dynamics:

\begin{thm}\label{thm:Henon Entropy}
Any transcendental H\'enon map has infinite topological entropy.
\end{thm}

As an immediate corollary we obtain an alternative proof that the Julia set is non-empty, and by the Variational Principle that the metric entropy is also infinite. The proof implies that a transcendental H\'enon map has infinitely many periodic cycles of any order greater than $2$. This result gives a complete description on the possible periodic cycles, since there exist transcendental H\'enon maps without any periodic cycles of orders $1$ and $2$~\cite{henon2}. We recall the analogy with one-dimensional transcendental functions, which may not have any fixed points, but always have infinitely many periodic cycles of any order greater than $1$.

\medskip

The topological entropy of holomorphic maps is a topic with an interesting history. It was shown by Gromov that the topological entropy of a rational function of degree $d$ is $\log(d)$, a result written in a preprint in 1977, but not published until 2003 \cite{gromov}. In the meantime the result was obtained independently by Lyubich \cite{lyubich}.

Smillie~\cite{smillie} proved in 1990 that a polynomial H\'enon map of degree $d$ has topological entropy $\log(d)$. Preliminary results for transcendental H\'enon maps were obtained by Dujardin~\cite{Dujardin04}, who proved that the entropy of a \emph{H\'enon-like map} of degree $d$ is $\log(d)$ as well, and used this fact to construct examples of transcendental H\'enon maps with infinite topological entropy.

The fact that transcendental functions in one complex variables always have infinite entropy was proved in the paper \cite{transcendentalentropy} by the three last authors. However, after completing our paper we learned that this result was obtained earlier by Markus Wendt~\cite{WendtDiploma, Wendt, WendtMan}, who never published this work. The proof we present in this paper will closely follow ideas from the proof of Wendt.

\subsection{Outline of the proof}

Following Wendt we give  different proofs depending on whether the family of rescaled maps $f_n(z) := f(n\cdot z)/n$ is \emph{quasi-normal} or not (see Definition~\ref{defn:quasinormality}). If this family is quasi-normal, Wendt showed that $f$ acts as a polynomial-like map of arbitrarily large degree on larger and larger domains, hence has infinite entropy. Similarly, we show that $F$ acts as a H\'enon-like map of arbitrarily large degree, hence by Dujardin's result $F$ also has infinite entropy.

When the family $(f_n)$ is not quasi-normal, Wendt shows that one can find an arbitrarily large number of disks with pairwise disjoint closures, such that each of these disks contains a univalent preimage of all but at most $2$ of the disks; a consequence of the Ahlfors Five Islands Theorem~\cite{Bergweiler}. In the H\'enon setting, we prove similarly that any suitable graph over each of these disks contains a preimage of a suitable graph over all but at most $2$ of the other disks. In both the quasi-normal and the non quasi-normal setting we obtain completely invariant compact subsets on which the entropy is arbitrarily large. It follows that the topological entropy is infinite.

\medskip

In section \ref{section:prelim} we recall background on topological entropy, including the definition of entropy on non-compact spaces that we will use. We also discuss the notion of quasi-normality, and recall Ahlfors Five-Islands Theorem and some of its consequences. In section \ref{section:proof} we prove Theorem \ref{thm:Henon Entropy}, first under the assumption that the family $(f_n)$ is quasi-normal, and then under the assumption that the family is not quasi-normal. In section \ref{section:periodic} we prove the existence of periodic cycles of any period at least $3$.  In section \ref{section:lacunary} we construct examples of transcendental H\'enon maps with arbitrarily slow or fast growing entropy in terms of the size of the compact sets.

\medskip

\noindent {\bf Acknowledgment.} The result obtained here answers a question asked to us by both Romain Dujardin and Nessim Sibony. We are grateful for their suggestion, which stimulated this research. The proof of our result closely follows the ideas of Markus Wendt in unpublished work. We are grateful for Walter Bergweiler for bringing this work to our attention, and for further discussion on this topic.

\section{Preliminaries}\label{section:prelim}
\subsection{Entropy}

For maps acting on compact spaces the concept of topological entropy has been introduced in \cite{Adler}.
\begin{defn}[Definition of topological entropy for compact sets]\label{entropy compact}
Let $f:X \rightarrow X$ be a continuous self-map of a compact metric space $(X, d)$.
 Let   $n \in \mathbb N$ and $\delta > 0$.
A  set $E\subset X$ is called \emph{$(n,\delta)$-separated}  if   for any $z\neq w\in E$ there exists $k\leq n-1$ such that
$d(f^k(z),f^k(w))> \delta$.
Let $K(n, \delta)$ be the maximal cardinality of an $(n,\delta)$-separated set.
Then the \emph{topological entropy} $\htop(X,f)$ is defined as
$$
\htop(X,f):=\sup_{\delta>0}\left\{\limsup_{n\ra\infty}\frac{1}{n}\log K(n,\delta)\right\}.
$$
\end{defn}
In the literature there are several non-equivalent  natural generalizations for the definition of topological entropy on non-compact spaces (see for example
\cite{Bowen}, \cite{Bowen71}, \cite{Bowen73}, \cite{Hofer}, and more recently \cite{Hasselblatt}). We will use the definition introduced by \cite{CanovasRodriguez} which is smaller than or equal to all the ones mentioned above.
\begin{defn}
Let $f:Y \rightarrow Y$ be a continuous self-map of a  metric space $(Y, d)$. Then
 the \emph{topological entropy} $\htop(Y,f)$ is defined as the supremum of $\htop(X,f)$ over all forward invariant compact subsets $X\subset Y.$ If there is no forward invariant compact subset the topological entropy is defined to be $0$.
\end{defn}
\begin{remark}
Notice that this definition does not depend on the metric inducing the topology on $Y$, and is invariant by topological conjugacy, hence the name ``topological entropy'' is justified. Notice also that in [REF] the last three named authors used a slightly different definition of topological entropy, a priori  larger than or equal to the above one.
\end{remark}

\begin{lemma}
Assume that the map $f\colon Y\to Y$ is injective. Then $\htop(Y,f)$ is equal to  as the supremum of $\htop(X,f)$ over all {\sl completely} invariant compact subsets $X\subset Y.$
\end{lemma}
\begin{proof}
Let $X$ be a compact forward invariant subset of $Y$. Consider  the compact set
 $\Lambda:=\bigcap_{n\geq 0}f^n(X).$
 Since $f$ is injective, it follows that the map $$f|_\Lambda\colon \Lambda\to \Lambda$$ is bijective, and in particular $\Lambda$ is completely invariant by $f$. The following classical result yields the lemma.
\end{proof}
\begin{theorem}
Let $g:K \rightarrow K$ be a continuous self-map of a compact metric space $(K, d)$ and let $\Lambda :=\bigcap_{n\geq 0}g^n(K).$ Then
$$\htop(K,f)= \htop(\Lambda, f).$$
\end{theorem}
For the proof, see e.g. Block and Coppel.

\subsection{Ahlfors Theorem and quasinormality}

The following is a version of Ahlfors five islands Theorem which can be found in  \cite{Bergweiler}, Theorem A.1. A more classical formulation of Ahlfor's five islands theorem and Corollary~\ref{cor:3 island}  in terms of  regularly exhaustible Riemann surfaces can be found in  \cite{Schiff}, Chapter 1.9.
\begin{thm}[Ahlfors five islands Theorem]
Let $D_1, \ldots, D_5$ be Jordan domains on the Riemann sphere with pairwise disjoint closures and let $D\subset\C$ be a domain. Then the family of all meromorphic functions $f:D\ra\hat{\C}$ with the property that none of the $D_j$ has a univalent preimage in $D$ is normal.
\end{thm}
As observed in \cite{Bergweiler} after the statement of Theorem B.3, if the functions are holomorphic on  $D$ and the domains $D_i$ are bounded the number 5 can be replaced by 3.
\begin{corollary}\label{cor:3 island}Let $D_1, \ldots, D_k$ with $k\geq3$ be bounded Jordan domains on the Riemann sphere with pairwise disjoint closures and let $D\subset\C$ be a domain.
Let $\FF$ be   family of holomorphic  functions $f:D\ra\hat{\C}$ which is not normal in $D$. Then for all but at most 2 values of $j$,  $D_j$ has a univalent preimage in $D$.
\end{corollary}

We recall the  definition of quasi-normality from the Appendix in \cite{Schiff}.
\begin{definition}\label{defn:quasinormality}
Let $\Omega\subset \C$ be a domain. A family  $\mathcal{F}$ of holomorphic functions on $\Omega$ is {\sl quasi-normal} if for every sequence $(f_n)$ of functions in $\Omega$ there exists a finite set $Q\subset \Omega$ and a subsequence $(f_{n_k})$ of $(f_n)$ which converges uniformly on compact subsets of $\Omega\setminus Q$.
\end{definition}

The rest of this subsection is devoted to the proof of  the following Proposition~\ref{prop:non normal subsequences}, which in turn will be used in the proof of the not quasi-normal case.

\begin{prop}\label{prop:non normal subsequences}
Let $\Omega\subset \C$ be a domain and let $\FF$ be a not quasi-normal family of holomorphic functions $\Omega\to \C$. Then there exists a sequence  $(f_n)
\subset\FF$ and an infinite subset $Q =(x_j)_{j\geq 1}\subset \Omega$ such that
no subsequence of $(f_n)$ converges uniformly in any neighborhood of any $x_j$.
\end{prop}

\begin{lemma}\label{lem:quasinormality1}
Let $\Omega\subset \C$ be a domain and let $\FF$ be a not quasi-normal family of holomorphic functions $\Omega\to \C$.
Then there exist a sequence $(f_n)$ in $\mathcal{F}$
with the following property: for every subsequence $(f_{n_k})$, there exists an infinite set $E(f_{n_k})\subset \Omega$ such that $(f_{n_k})$ is not normal in any neighborhood of a point in $E(f_{n_k})$.
\end{lemma}
\begin{proof}
Assume  $\mathcal{F}$ is not quasi-normal. Then there exists a sequence  $(f_n)$ in $\mathcal{F}$ such that for any finite set $L\subset \Omega$ and every subsequence $(f_{n_k})$ of $(f_n)$, $(f_{n_k})$  does not converge uniformly on compact subsets in $\Omega\setminus L$.
For every subsequence $(f_{n_k})$, define $E(f_{n_k})$ as the set of all points $x$ in $\Omega$ such that the sequence $(f_{n_k})$ is not normal in any neighborhood  of $x$. We just need to prove that $E(f_{n_k})$ is not a finite set. If by contradiction $E(f_{n_k})$ is a finite set, then for all points $y\in \Omega\setminus E(f_{n_k})$, the sequence $(f_{n_k})$ is locally normal around $y$. Since normality is a local property, it  follows that $(f_{n_k})$ is normal on $\Omega\setminus E(f_{n_k})$, and thus we can extract a subsequence of $(f_{n_k})$ converging on $\Omega\setminus E(f_{n_k})$, which is a contradiction.
\end{proof}
\begin{lemma}\label{lem:quasinormality2}
Let $\Omega\subset \C$ be a domain and let $x\in \Omega$.
If a sequence of holomorphic functions $(f_n\colon \Omega\to \C)$ is not normal in any neighborhood of $x$, then we can extract a subsequence $(f_{n_k})$ with the property that no subsequence of  $(f_{n_k})$  converges uniformly  in any neighborhood of $x$.
\end{lemma}
\begin{proof}
 Recall that a sequence $(f_n)$ is normal if and only if it is equicontinuous with respect to the spherical metric on the Riemann sphere. Since $(f_n)$ is not normal on any neighborhood of $x$, it follows that $(f_n)$ is not equicontinuous in $x$.
This means that there exists a constant $\varepsilon>0$ such that for all $j$ there exist $|x_j-x|< 1/j$ and an integer $n_j$ such that
$$d(f_{n_j}(x_j), f_{n_j}(x))\geq \varepsilon.$$ But then the sequence $(f_{n_j})$ cannot have a subsequence converging uniformly in any neighborhood of $x$.
\end{proof}

\begin{proof}[Proof of Proposition~\ref{prop:non normal subsequences}]
Let $(f_n)$ be the sequence given by  Lemma~\ref{lem:quasinormality1}, and $E(f_n)$ be the associated non-normality infinite set.
Choose $x_1\in E(f_n)$. By  Lemma~\ref{lem:quasinormality2} there exists a subsequence $(f_{n_1(h)})$ of $(f_n)$ such that every subsequence of $(f_{n_1(h)})$ does not converge in any neighborhood of $x_1$.

Let now  $E(f_{n_1(h)}))$ be the infinite set given by  Lemma~\ref{lem:quasinormality1} for the subsequence $(f_{n_1(h)})$. Choose $x_2\in E((f_{n_1(h)}))$ different from $x_1$.
By Lemma~\ref{lem:quasinormality2} there exists a subsequence $(f_{n_2(h)})$ such that every subsequence of $(f_{n_2(h)})$ does not converge uniformly in any neighborhood of the points  $x_1,x_2$. By induction we obtain
an infinite set $Q:=(x_j)_{j\geq 1}$ and a family $((f_{n_k(h)}))_{k\geq 1}$ of nested subsequences of $(f_n)$
such that for all $k\geq 1$ no subsequence of  $(f_{n_k(h)})$ converges uniformly in any neighborhood of the points $x_1, \dots, x_k$. The diagonal subsequence $(g_h:=f_{n_h}(h))$ gives the result.
\end{proof}

\section{Proof of Theorem~\ref{thm:Henon Entropy}}\label{section:proof}

Let $F(z,w)=(f(z)-\delta w,z)$ be a transcendental H\'enon map.
For $n\in\N$ and $z\in\C$ let us define
$$
f_n(z):=\frac{f(nz)}{n}.
$$
Observe that for each $n,$ $f$ and $f_n$ are topologically conjugate via the map $z\mapsto nz$, so they have the same entropy.
Analogously, the maps $F_n(z,w)=(f_n(z)-\delta w,z)$ are topologically conjugate to $F$ and hence  have the same entropy as $F$.

\begin{example}
For $f(z) = e^z$ the functions $f_n$ diverge on the right half plane, and converge to $0$ on the left half plane, thus $(f_n)$ is not quasi-normal in any neighborhood of any point on the imaginary axis.

Consider  a  sequence of complex numbers $(a_\ell)$ with $|a_\ell|\rightarrow \infty$ and  $|a_{\ell+1}/a_\ell| \rightarrow \infty$, and define
$$
f(z) = \prod_{\ell\ge 1} (1 - z/a_\ell).
$$
Since the infinite product converges for every $z$ by choice of the $a_\ell$,
and since it is not a polynomial,  $f$ is a transcendental entire function.   Notice that $f_n(0)\ra0$, that the zeros of $f$ are $\{a_\ell\}_{\ell\geq1}$, and that the zeros of $f_n$ are $Z_n:=\{a_\ell/n\}_{\ell\geq1}$.
%z_\ell(n_j):=

Given any sequence in $(f_n)$ we can find a subsequence $(f_{n_j})$ for which the sets of zeros $Z_{n_j}=\{a_\ell / n_j\}_{\ell \ge 1}$ converge as $n_j\ra\infty$ to the set $Z_\infty$, which is either $\{0, \infty\}$ or  $\{0, \infty, q\}$ for some $q\in \mathbb C\setminus \{0\}$, in terms of the Hausdorff metric on the Riemann sphere.% $\hat{\mathbb C}$.%, {\color{blue} i.e. that for every neighborhood $U$ of $Z_\infty$ in $\widehat \C$ there exists $j_0$ such that for all $j\geq j_0$ we have $Z_{n_j}\subset U$.}

Indeed, if a sequence of zeros $a_{\ell_j}/n_j$ accumulates on a point $q\neq0, \infty$, then up to passing to a subsequence we may assume that $a_{\ell_j}/n_j \ra q$ as $j\ra\infty$. Since $|a_{j+1}/a_j| \rightarrow \infty$ it follows that as $j\to\infty$  we have that  $a_{i_j}/n_j$ tends to $0$ whenever $i_j<\ell_j $, and converges to $\infty$ whenever $i_j>\ell_j$.

Let us work with the case $Z_\infty=\{0, \infty, q\}$. Write $f_{n_j}(z)$ as a product of three terms as follows:
$$
f_{n_j}(z) = \left(\frac{1}{n_j} \prod_{\ell<\ell_j} \left(1 - \frac{z n_j}{a_\ell}\right)\right) \left(1 - \frac{z n_j}{a_{\ell_j}}\right)\left(\prod_{\ell> \ell_j} \left(1 - \frac{z n_j}{a_\ell}\right)\right).
$$
Observe that on any compact subset of $\mathbb C  \setminus \{0,q\}$ the second of these terms converges uniformly to the non-zero function $1-z/q$, while the third term converges uniformly to the constant function $1$. The first term diverges uniformly, proving quasi-normality. In the case $Z_\infty = \{0, \infty\}$ one writes $f_{n_j}(z)$ as a product of two terms, similarly obtaining locally uniform divergence on $\mathbb C \setminus \{0\}$.
\end{example}

The proof of Theorem ~\ref{thm:Henon Entropy} is divided into two cases, with different proofs,  depending on whether $\FF:=(f_n)$ is a quasi-normal family or not. As mentioned in the introduction, the outline of our proof follows Wendt's proof \cite{WendtDiploma,Wendt,WendtMan} for  the one-dimensional case.

\subsection{Quasinormal Case}
In this subsection we prove the following result:
\begin{thm}\label{thm:qn2D}Let $F:(z,w) \mapsto (f(z) - \delta w, z)$ be a transcendental H\'enon map, and suppose that the transcendental functions defined by $f_n(z) = f(nz)/n$ form a quasi-normal family. Then $F$ has infinite entropy.
\end{thm}

For any $r\in\R$ let us denote by $\D_r$ the Euclidean disk of radius $r$ centered at $0$. Let $f$ be entire transcendental and let $\FF$ be the family of rescalings $f_n(z)=f(nz)/n$.
Assume that  $\FF$ is quasi-normal.
Then there is a subsequence $(f_{\nk})$ of $(f_n)$ and a finite set $Q$ such that $(f_\nk)$ converges uniformly on compact sets of $\C\setminus Q$.

\begin{lem}\label{lem:tending to infinity near 0}
The set $ Q$ contains the origin, and  there exists $0<s<1$ such that $f_{\nk}\ra\infty$ uniformly on compact subsets of $\D_s\setminus\{0\}$.
\end{lem}

\begin{proof}
Observe first that for every  $r>0$, any subsequence of $(f_n)$ is unbounded in the circle $\partial\D_r$.
Indeed,  for any $n$ we have that  $f_n(\D_{1/\sqrt{n}})=f(\D_{\sqrt{n}})/n$, and the maximum modulus of a transcendental function on a disk of radius $r$ grows faster than   $r^2$.

We claim  that $(f_\nk)$ does not converge uniformly in a neighborhood of $0$, so in particular,  $0\in Q$.  Indeed, $f_\nk(0)=f(0)/\nk \ra0$ as $n_k\ra\infty$, while  $(f_\nk)$ is unbounded in any neighborhood of $0$.
Since $Q$ is finite we can find $s$ such that $f_{\nk}\ra g$ uniformly on compact subsets of $\D_s\setminus\{0\}$, with $g:\D_s\setminus\{0\}\ra \C$ or $g=\infty$.  Since  $(f_\nk)$ is unbounded in any circle $\partial\D_r$ we obtain $g=\infty$.
\end{proof}

\begin{prop}\label{prop:qn1D}
Let $ s, (f_{n_k})$  be as in Lemma~\ref{lem:tending to infinity near 0}.
Let $0<r<s$, and let $R>0$ and $m\in \N$. Then there exists $k_0\in\N$ such that for $k>k_0$ we have
 \begin{enumerate}
 \item $|f_{n_k}(z)|>R$ for every $z\in \partial \D_r$,
\item  the winding number  of the curve $f_{n_k}(\partial \D_r)$ around the origin is  larger than or equal to $m$.
  \end{enumerate}
  \end{prop}

\begin{proof}

 (1) is an immediate consequence of Lemma~\ref{lem:tending to infinity near 0}.
We now prove (2).
Let $a\in\D_R$ be a non-exceptional point for $f$.
Fix $m\in\N$, and let $\rho=\rho(m)$ such that $a$ has at least $m$ preimages in $\D_\rho$ under  $f$. Let $M$ such that $f(\D_\rho)\Subset \D_M$.
It follows that there is  a connected component  $W$    of $f^{-1}(\D_M)$ which contains $\D_\rho$, and hence contains at least $m$ preimages of $a$ under $f$.

Let $k_0$ be large enough such that for all $k\geq k_0$ we have $M/\nk <R$, and such that (1) holds. Let $k\geq k_0$.
Denote by $W/\nk$ the set $\{z/\nk \colon z\in W\}$. Then if $z\in  W/\nk$ we have $\nk z\in W$ and hence $|f_\nk(z)|<R.$
Thus $W/n_k \subset f_\nk^{-1}(\D_R)$.
Notice that $0\in W/n_k$.  It follows by (1) that $W/n_k \subset \D_r$.

We now claim that $W/\nk$ contains at least $m$ preimages of $a_k := a/n_k$ under $f_\nk.$ Indeed $W$ contains at least $m$ preimages of $a$ under $f$, and for any such preimage $z$ we have that
$$
f_\nk(\frac{z}{\nk})=\frac{f(z)}{\nk}=a_k.
$$
Since $a_k\in \D_R$, the  result follows by the argument principle.
\end{proof}

%\begin{corollary}Since the entropy of polynomial-like maps of degree $d$ equals $\log d$, it follows that if $\FF$ is quasi-normal then $f$ has infinite topological entropy.
%\end{corollary}
%
%\begin{rem}Proposition~\ref{prop:qn1D} gives that $f_{n_k}:U_k\ra\D_R$ is polynomial-like of degree $d_k$. Since $f_{n_k}(z)=f(n_k z)/n_k$, this implies that $f:  n_k U_k\ra \D_{{n_k}R} $ is polynomial-like of degree $d_k$.
%\end{rem}

%\subsubsection{Henon-like maps and Proof of Theorem~\ref{thm:qn2D}}

%The following results and definitions are from  \cite{Dujardin04}.

Let $\Delta=\D_{r_1}\times \D_{r_2}$ be a bidisk, $\partial_v\Delta,\partial_h\Delta $ denote its vertical and horizontal boundary respectively. The following definition of H\'enon-like maps is  Definition 2.1 in \cite{Dujardin04}.

\begin{defn}[H\'enon-like map]\label{defn:henon-like}
An injective holomorphic map $H$ defined in a neighborhood of $\overline\Delta$ is called  \emph{H\'enon-like} if
\begin{enumerate}
\item $H(\Delta)\cap \Delta\neq\varnothing$;
\item $H(\partial_v(\Delta))\cap\ov{\Delta}=\varnothing$;
\item $H(\ov{\Delta})\cap\partial \Delta\subset \partial_v(\Delta) $.
\end{enumerate}
\end{defn}

%As observed in \cite{Dujardin04} right after the definition 2.1, one can replace the straight polydisks $\D_{r_1}\times \D_{r_2}$ with more general sets. In this paper we will consider (generalized) polydisks of the form  $\Delta=U\times \D_\epsilon \subset \C^2$ where  $U\subset \C$ is a Jordan domain with analytic boundary.

Let $\pi_z, \pi_w:\C^2\ra\C$ denote the projection to the $z$ and to the $w$ axis respectively.

\begin{defn}[Degree of a H\'enon-like map]\label{dujdef}
Let $H$ be a H\'enon-like map defined in a neighborhood of  $\overline\Delta=\overline\D_{r_1}\times \overline\D_{r_2}$ and let $L_h$ be any horizontal line intersecting $\Delta$.
 Consider the holomorphic function
\begin{equation}\label{eq:proper projections}
 \pi_z\circ H: H^{-1}(\Delta)\cap \Delta\cap L_h\ra \D_{r_1}.
\end{equation}
Then by condition (3) of Definition \ref{defn:henon-like} we have that if $(z,w)\in \partial (H^{-1}(\Delta)\cap \Delta\cap L_h)$, then $H(z,w)\in \partial_v\Delta$, which means that
the function  in (\ref{eq:proper projections}) is proper, and thus a branched covering.
 By Proposition 2.3 in \cite{Dujardin04}, its degree is independent of the chosen horizontal line. This integer is the {\sl degree} of the H\'enon-like map $H$.
 %the open set $H^{-1}(\Delta)\cap \Delta\cap L_h$ has a finite number of connected components, hence its topological degree is well defined and is equal to the degree of $\pi_z\circ H$ as a branched covering.
\end{defn}

The following theorem is proved in \cite[Theorem 3.1]{Dujardin04}.
\begin{thm}\label{thm:Henonlike entropy}
Let $H$ be a H\'enon-like map of degree $d$. The topological
entropy of $H$ is $\log d$.
\end{thm}

\begin{lemma}\label{polynomial->henon}

Let $f$ be a holomorphic function defined in a neighborhood of $\overline{\D_r}$, let $\delta \neq 0$, and suppose that $|f(z)| > (|\delta|+1)\cdot r$ whenever $|z| = r$. Assume that the winding number of the curve  $f(\partial\D_r)$  around the origin  is $d \ge 1$. Then the map $F:(z,w) \mapsto (f(z) - \delta w, z)$ is a H\'enon-like map of degree $d$ on $\overline\Delta = \overline\D_r \times \overline\D_r$.
\end{lemma}
\begin{proof}
We check the three properties in Definition \ref{defn:henon-like}. The estimate $|f(z)| > (|\delta|+1)\cdot r$ gives that $|f(z) - \delta w| > r$ for all $(z,w) \in \partial_v\Delta$, which implies property (2). The formula for $F$ therefore implies that $F(\overline \Delta)$ cannot intersect $\partial_h\Delta$, giving property (3).  Since $f(\partial \D_r)$  winds around $0$ exactly $d\geq1$ times, $0$ has at least one preimage $a\in\D_r$. Hence $F(a,0)=(0,a)\in \Delta$ which gives Property (1).

We now show that $F$ has degree $d$ on $\overline \Delta$. By Definition \ref{dujdef} it is enough   to  show that  $0\in\D_r$  has   $d$ preimages counted with multiplicity in $ F^{-1}(\Delta)\cap \Delta\cap L_0$ under $\pi_z\circ F$, where $L_0$ is the horizontal line passing through $0$. It is easy to see that these points coincide with the preimages in $\D_r$ of the origin under the function $f$, and the result follows by the argument principle since the curve $f(\partial\D_r)$ winds $d$ times around $0$.
%Since $f(\partial\D_r)$ winds $d$ times around $0$, the latter has  $d$ preimages  %counted with multiplicity
%under $f$ in $\D_r$. Let $a$ be any such preimage and consider the point $P:=(a,0)\in \Delta \cap L_0$.
%We then have that
%$$
%\pi_z\circ F(a,0)=\pi_z(f(a),a)=0,
%$$
%which gives $d$ preimages of $0$ under $\pi_z\circ F$ contained in $\Delta\cap L_0$. % (counted with multiplicity).
% It only remains to show that any such point $P$ is contained in $F^{-1}(\Delta)$, or equivalently, that $F(P)=(f(a),a)=(0,a)\in\Delta$; the latter is true because $a\in \D_r$.
\end{proof}

\begin{proof}[Proof of Theorem~\ref{thm:qn2D}]
Recall that
$
F_n(z,w):=(f_n(z)- \delta w, z),
$
and that $F_n$ is topologically conjugate to $F$ for all $n\geq 0$.

Fix $m\in \N$. Let  $s, (f_\nk)$  be as in Lemma~\ref{lem:tending to infinity near 0} and fix $r<s, R>(|\delta|+1)r$.
Let $k_0$ be given by Proposition  ~\ref{prop:qn1D}. Then, if $k\geq k_0$, it follows by Lemma~\ref{polynomial->henon} that
$F_{n_k}$ is H\'enon-like of degree at least $m$ on the bidisk $\overline\D_r\times\overline\D_r$.
By  Theorem~\ref{thm:Henonlike entropy} we have that the entropy of $F_{n_k}$ is larger than or equal to $\log m$, and by topological invariance the same holds for the map $F$.
\end{proof}
\subsection{Non Quasinormal Case}
We will now prove the following:
\begin{theorem}\label{thm:non-quasi-normal}
Let $F:(z,w) \mapsto (f(z) - \delta w, z)$ be a transcendental H\'enon map, and suppose that the transcendental functions defined by $f_n(z) = f(nz)/n$ do not form a quasi-normal family. Then $F$ has infinite entropy.
\end{theorem}

\subsubsection{Proof of Theorem~\ref{thm:non-quasi-normal}}\label{Sect:proof non qn}

Assume that the family $(f_n)$ is not quasi-normal.
Let $(f_{n_h})$ be the subsequence of $(f_n)$ given by Proposition~\ref{prop:non normal subsequences} and let $Q=(x_j)_{j\geq 1}$ be the associated infinite set.
Fix $k\geq 1$. Let $R>0$ be such that the closures of the disks $\D_R(x_j)$, for $j = 1, \ldots, k$ are pairwise disjoint.  Next define $0<r<R$ such that $|\delta| r < R-r$.
Recall that  no subsequence of $(f_{n_h})$ is normal in any of the $k$ disks  $\D_r(x_j),$ $j=1,\dots, k$.
\begin{lemma}\label{thefunction}
For a given $n_h$, and    for    $i,\ell\in\{1, \ldots , k\}$   let
 $$
 J(i,\ell):=\{j\in \{1,\dots, k\}: \text{$\D_R(x_j + \delta x_\ell)$ admits a biholomorphic preimage under $f_{n_h}$ in  $\D_r(x_i)$}\}.
 $$
 Then there exists $n_h$ such that
 $\, \#( J(i,\ell))\geq k-2$ for every   $i,\ell\in\{ 1,\ldots,k\}$.
\end{lemma}
\begin{proof}
Assume by contradiction that this is not the case. Then for all $n_h$ there exist $i,\ell\in \{1,\dots, k\}$ and $3$ distinct values  $j_1, j_2, j_3\in 1,\dots, k$ such that
 the disks $\D_R(x_{j_1} + \delta x_\ell),\D_R(x_{j_2} + \delta x_\ell), \D_R(x_{j_3 }+ \delta x_\ell)$ do not admit biholomorphic preimages via $f_{n_h}$ in the disk $\D_r(x_i)$. It follows that we can find a subsequence  $(f_{m_h})$ with the following property: there exist $i,\ell\in 1,\dots, k$ and $3$ distinct values  $j_1, j_2, j_3\in\{ 1,\dots, k\}$ such that for all $ m_h$ the disks $\D_R(x_{j_1} + \delta x_\ell),\D_R(x_{j_2} + \delta x_\ell), \D_R(x_{j_3} + \delta x_\ell)$ do not admit biholomorphic preimages via $f_{m_h}$ in the disk $\D_r(x_i)$. By Ahlfors five islands Theorem (see Corollary~\ref{cor:3 island}) $(f_{m_h})$ is normal in $\D_r(x_i)$, which gives a contradiction.
\end{proof}

In what follows we denote the map $f_{n_h}$ given by the previous lemma simply as $f_n$.
We will consider the dynamics of the H\'enon map $F_n(z,w) := (f_n(z) - \delta w, z)$, which is linearly conjugate to $F$.

%As usual, let us denote by $\pi_z,\pi_w:\C^2\ra\C$ the projection to the first and to the second coordinate respectively.
\begin{definition}\label{defn:disks}
Let $i,\ell$ both lie in $\{1, \dots, k\}$.
A holomorphic disk $D$ is called a $(i,\ell)$-disk if
\begin{itemize}
\item it is a holomorphic graph over $\D_r(x_i)$, that is $D$ can be parametrized as $(z,w(z))$ with $w(z)$ holomorphic in $\D_r(x_i)$;
\item $\pi_w(D)\subset \D_r(x_\ell)$, where $\pi_w$ is the projection to the second coordinate.
\end{itemize}
\end{definition}
\begin{lemma}\label{lem:transitive disks}
Let $i,\ell\in\{1, \dots, k\}$. Then for all $j\in J(i,\ell)$
and for any $(i,\ell)$-disk $D$ there exists a holomorphic disk $V \subset D$ for which $F_n(V)$ is a $(j,i)$-disk.
\end{lemma}

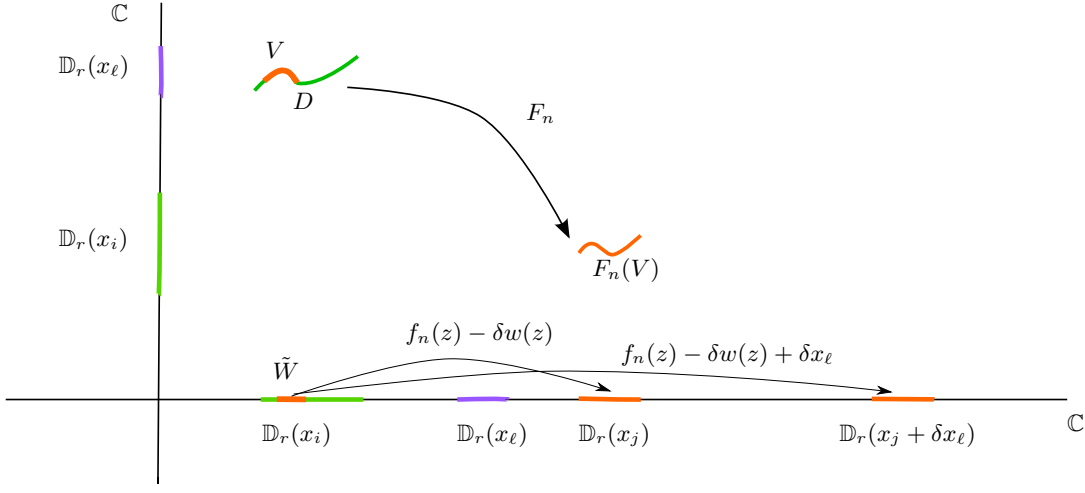
\begin{figure}[htb!]
\centering
\begingroup%
  \makeatletter%
    \setlength{\unitlength}{400bp}%
  \makeatother%
  \begin{picture}(1,0.45364384)%
    \put(0,0){\includegraphics[width=\unitlength]{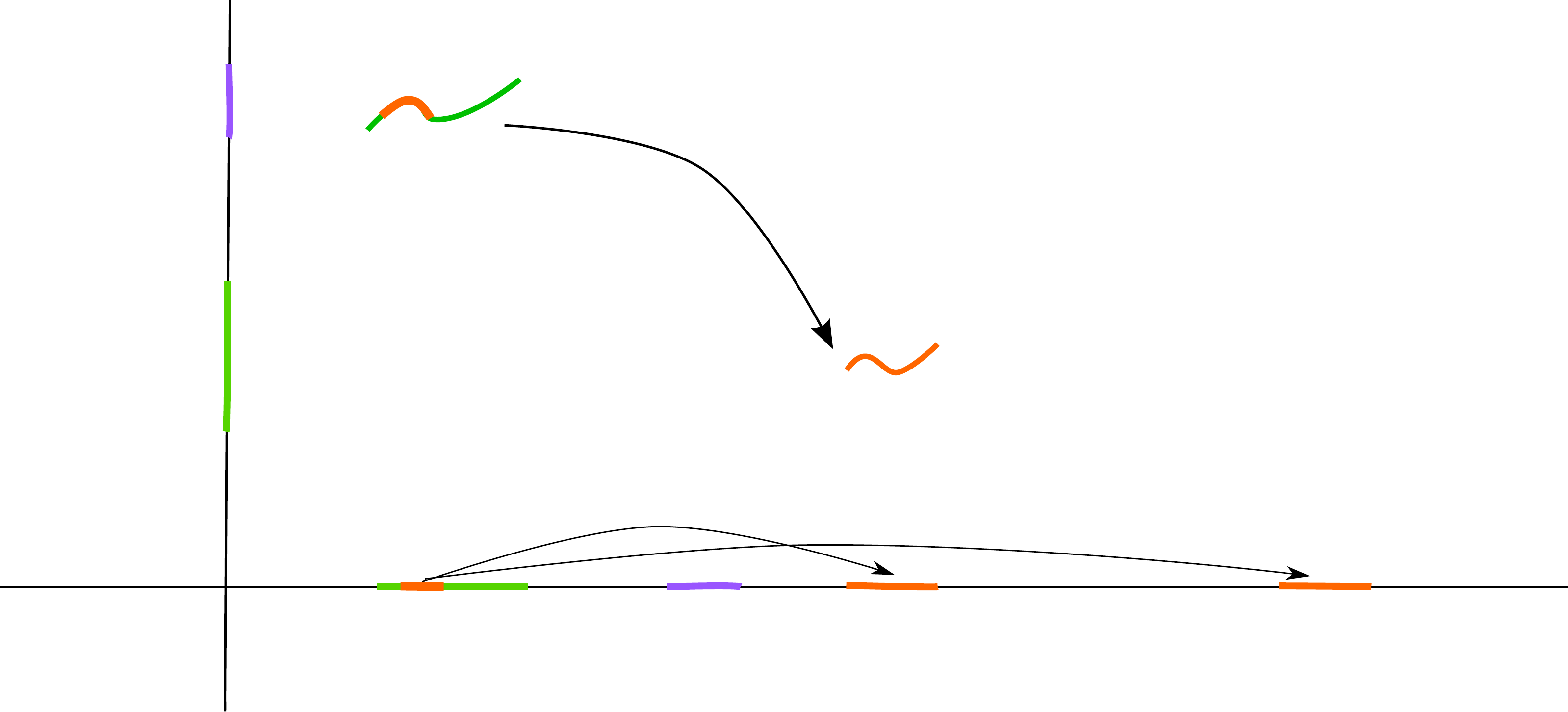}}%
    \put(1,0.05){\color[rgb]{0,0,0}\makebox(0,0)[lb]{\small{$\C$}}}%
    \put(0.09930165,0.43645373){\color[rgb]{0,0,0}\makebox(0,0)[lb]{\small{$\C$}}}%
    \put(0.24161921,0.035){\color[rgb]{0,0,0}\makebox(0,0)[lb]{\small{$\D_r(x_i)$}}}%
    \put(0.4253885,0.035){\color[rgb]{0,0,0}\makebox(0,0)[lb]{\small{$\D_r(x_\ell)$}}}%
    \put(0.54007161,0.035){\color[rgb]{0,0,0}\makebox(0,0)[lb]{\small{$\D_r(x_j)$}}}%
    \put(0.05,0.21873111){\color[rgb]{0,0,0}\makebox(0,0)[lb]{\small{$\D_r(x_i)$}}}%
    \put(0.05,0.38177454){\color[rgb]{0,0,0}\makebox(0,0)[lb]{\small{$\D_r(x_\ell)$}}}%
    \put(0.78536395,0.035){\color[rgb]{0,0,0}\makebox(0,0)[lb]{\small{$\D_r(x_j+\delta x_\ell)$}}}
    \put(0.49032953,0.34064557){\color[rgb]{0,0,0}\makebox(0,0)[lb]{\small{$F_n$}}}%
    \put(0.27063541,0.35308108){\color[rgb]{0,0,0}\makebox(0,0)[lb]{\small{$D$}}}%
    \put(0.55250712,0.19){\color[rgb]{0,0,0}\makebox(0,0)[lb]{\small{$F_n(V)$}}}%
    \put(0.24486377,0.40169867){\color[rgb]{0,0,0}\makebox(0,0)[lb]{\small{$V$}}}%
    \put(0.37676234,0.12959677){\color[rgb]{0,0,0}\makebox(0,0)[lb]{\small{$f_n(z)-\delta w(z)$}}}%
    \put(0.57945751,0.11047349){\color[rgb]{0,0,0}\makebox(0,0)[lb]{\small{$f_n(z)-\delta w (z)+\delta x_\ell$}}}%
    \put(0.25170296,0.09930899){\color[rgb]{0,0,0}\makebox(0,0)[lb]{\small{$\tilde{W}$}}}%
  \end{picture}%
\endgroup%
\caption{\small Illustration of the statement and proof of Lemma~\ref{lem:transitive disks}. The disks $\D_r(x_i)$ are contained in larger disks  $\D_R(x_i)$, which do not appear in this picture.}
\label{fig:basicregions}
\end{figure}

\begin{proof}
It is clear that the $w$-coordinates of $F_n(V)$ are contained in $\D_r(x_i)$, regardless of the choice of $V \subset D$. We therefore merely need to find a holomorphic disk $V\subset D$ such that $F_n(V)$  is a  graph over  the disk $\D_r(x_j)$ in the $z$-coordinate.
 Since   $j \in J(i,\ell)$ there is a biholomorphic preimage $W \subset \D_r(x_i)$   of $\D_R(x_j + \delta x_\ell)$ under $f_n$. It follows that  the function $f_n-\delta x_\ell\colon W\to \D_R(x_j)$ is a biholomorphism as well.
Let $z\mapsto (z,w(z))$ be the graph parametrization of $D$. We claim that there exists an open subdomain $\tilde W\subset W$ such that $f_n(z)-\delta w(z)\colon \tilde W\to \D_r(x_j)$ is a biholomorphism. Once this is proved, setting $V:=D \cap (\tilde W \times \mathbb C)$ yields the result.
Notice that up to shrinking $R$ we can assume that  $f_n-\delta x_\ell\colon \overline W\to \overline\D_R(x_j)$  is a homeomorphism.
  For all $z\in \partial W$ we have
$$
|(f_n(z)-\delta w(z))- (f_n(z)-\delta x_\ell)|=|\delta| |x_\ell-w(z)|\leq|\delta| r<R-r
$$
by assumption, hence by  Rouch\'e's Theorem it follows that for every $u \in \D_r(x_j)$ there exists exactly one point $z \in W$ such that $f_n(z)-\delta w(z)=u$.
Setting   $\tilde W:= (f_n-\delta w)^{-1}(\D_r(x_j))$ we have that  $f_n-\delta w\colon \tilde W\to \D_r(x_j)$ is a biholomorphism.
\end{proof}

We conclude the proof of non quasi-normal case by showing that Lemma~\ref{lem:transitive disks} implies that the topological entropy of $F_n$ is at least $\log(k-2)$.

Define the compact subsets of $\C^2$ $$H:=\bigcup_{1\leq i,\ell\leq k}\overline\D_r(x_i)\times\overline\D_r(x_\ell),\quad L:=\bigcap_{m\geq 0} F_n^{-m}(H).$$
Clearly $L$ is forward $F_n$-invariant.
We say that a sequence $(i_0, i_1, i_2, \ldots) \in \{1, \ldots k\}^{\mathbb N}$ is \emph{admissible} if  $i_{m+1} \in J(i_m,i_{m-1})$ for every $m\geq 1$ and similarly, a finite word is admissible if it is the start of an infinite admissible sequence.
 Clearly,  for every admissible sequence $(i_0, i_1, i_2, \ldots) $, there exists a point $P \in L$ for which $F_n^m(P)$ lies in a $(i_{m+1}, i_{m})$-disk for all $m \ge 0$. Moreover for all $m\geq 0$ there are at least $k^2\cdot (k-2)^{m-2}$ admissible words  of length $m$.

Thus $L$ contains at least $(k-2)^m$ points with distinct symbolic representations, which are therefore $(m,\varepsilon)$-separated as soon as
$$\varepsilon <  {\rm   min}_{i,\ell}\,{\rm dist}(\overline\D_r(x_i),\overline\D_r(x_\ell)).   $$
 This proves the claim that $F_n: L \rightarrow L$ has topological entropy at least $\log(k-2)$, which in turn completes the proof of Theorem \ref{thm:non-quasi-normal}.

\section{Periodic cycles of arbitrary order} \label{section:periodic}

We continue to a consider transcendental H\'enon map $F$ of the form
$$
(z,w) \mapsto (f(z) - \delta w, z).
$$
In the previous paper \cite{henon2} we showed that when $\delta = -1$ the map $F$ may not have any fixed point or periodic orbits of period $2$, but if $F$ has neither, then it must have periodic points of order $4$. The proof of this fact relied upon algebraic manipulations of the equation $F^4(z,w) = (z,w)$. Using the techniques presented in the previous sections we can now obtain the following description.

\begin{thm}
A transcendental H\'enon map has infinitely many periodic cycles of any order $N \ge 3$.
\end{thm}
\begin{proof}
We consider again the family of rescaled transcendental functions $(f_n)$. We have shown that if this sequence is quasi-normal then appropriate restrictions of the H\'enon map $F$ act as H\'enon-like maps of larger and larger degrees. It was proved by Dujardin in \cite{Dujardin04}, Proposition 5.7, that a H\'enon-like map of degree $d$ has exactly  $d^N$ points which are fixed under $F^N$, counted with  multiplicity. It follows that if the family $(f_n)$ is quasi-normal then $F$ has infinitely periodic cycles of any period.

%{[REMOVE ]Let us therefore assume that the family $(f_n)$ is not quasi-normal, let us fix $N\ge 3$, and let $k > 2N-2$.
%Let $f_n$ be the function given by Lemma \ref{thefunction}, and define as before $F_n:=(f_n(z)-\delta w,z).$ Consider the $(i,\ell)$-disks constructed in Definition \ref{defn:disks}, for $i,\ell=1,\dots, k$. Recall from Lemma \ref{lem:transitive disks} that for any $i,\ell=1,\dots, k$ there exists a subset $J(i,\ell) \subset \{1, \ldots , k\}$ with  $\#( J(i,\ell))\geq k-2$ such that for any $j \in J(i,\ell)$, any $(i,\ell)$-disk $D_{i,\ell}$ contains a holomorphic disk $D$ which $F_n$ maps onto an $(j, i)$-disk.}

%{\color{blue} A straightforward counting argument shows that there exist sequences $(i_0, i_1, \ldots, i_{N-1}, i_{N} = i_0, i_{N+1}=i_1)$ such that $i_{j+1} \in J(i_j, i_{j-1})$ for $j = 1, \ldots , N$. Indeed,  there are $k^N$ possible sequences.  The number of sequences $(i_0, i_1, \ldots, i_{N-1}, i_{N} = i_0, i_{N+1}=i_1)$ which violate $i_{j+1} \notin J(i_j, i_{j-1})$ for some $j$  is at most $2 N \cdot  k^{N-1} $. The counting argument is concluded by the assumption $k > 2N$.
%}

Let us therefore assume that the family $(f_n)$ is not quasi-normal and  fix $N\ge 3$. Let $k > 3N-1$, and let $f_{n_h}$ be the function given by Lemma \ref{thefunction}. Since the subsequence $(n_h)$ plays no further role in this proof, we will just write $n$ instead of $n_h$, and write as before $F_n:=(f_n(z)-\delta w,z).$ Consider the $(i,\ell)$-disks constructed in Definition \ref{defn:disks}, for $i,\ell=1,\dots, k$. Recall from Lemma \ref{lem:transitive disks} that for any $i,\ell=1,\dots, k$ there exists a subset $J(i,\ell) \subset \{1, \ldots , k\}$ with  $\#( J(i,\ell))\geq k-2$ such that for any $j \in J(i,\ell)$, any $(i,\ell)$-disk $D_{i,\ell}$ contains a holomorphic disk $V$ which $F_n$ maps onto an $(j, i)$-disk.
We first claim that the number of  $N$-tuples $(i_0, i_1, \ldots, i_{N-1})$ with distinct entries satisfying
$$
i_{j+1} \in J(i_j, i_{j-1}),\quad j =0,\ldots,N-1,
$$
(where the indices are taken modulo $N$) tends to infinity as $k\ra\infty$. Indeed, the number of $N$-tuples whose entries are all distinct over $k$ symbols is $k\cdot(k-1)\cdot\ldots\cdot (k-N+1)$; on the other hand by Lemma \ref{lem:transitive disks}, the number of such $N$-tuples  which violate the condition  $i_{j+1} \in J(i_j, i_{j-1})$ in at least one index is at most $2N k\cdot(k-1)\cdot\ldots\cdot (k-N+2)$. Hence the number of admissible sequences is at least $k\cdot(k-1)\cdot\ldots\cdot(k-N+2)(k-3N+1)\ra\infty$ as $k\ra\infty$.
Notice that this counting argument breaks down for $N=2$, in agreement with the fact that there exists transcendental H\'enon maps without  periodic points of period $2$.

We will now argue that corresponding to any  sequence $\{(i_0, i_1), \ldots, (i_{N-1}, i_0)\}$ of length $N$ which is periodic in the sense discussed above  we can find a periodic cycle of minimal period $N$.
 
Observe that in the proof of Lemma \ref{lem:transitive disks} the holomorphic disk $V \subset D$ is of the form $D \cap (\tilde{W} \times \mathbb C)$, where $\tilde{W} \subset W$ depends on $D$,  but $W$
%depends (not necessarily uniquely) only on the three indices $i,j,\ell$ of the domain, the $(i,\ell)$-disk, and the co-domain, the $(j,i)$-disk. In particular $W$
is independent of the chosen $(i,\ell)$-disk $D$. Indeed, it is by construction a simply connected domain $W \in \D_r(x_i)$ that is mapped univalently onto $\D_R(x_j + \delta x_\ell)$ by the function $f_n$, hence it depends only on the three indices $i,j,\ell$ of the domain, the $(i,\ell)$-disk, and the codomain, the $(j,i)$-disk.

It follows that having chosen the domain $W$, the intersection of the bidisk $W \times \D_r(x_\ell)$ with the preimage $F_n^{-1}(\D_r(x_j)\times\D_r(x_i))$ is connected; a union of straight horizontal disks $V_w \subset W\times \{w\}$ for $w \in \D_r(x_\ell)$.

Let us now consider the periodic sequence $(i_0, i_1, \ldots, i_{N-1})$ discussed earlier, where each $i_{j+1} \in J(i_j, i_{j-1})$. For each triple $(i_{j-1}, i_j, i_{j+1})$ we select a disk $W_j \subset \D_r(x_{i_j})$ as above, for $j\ge N$ we define these sets inductively by $W_j = W_{j-N}$, obtaining a periodic sequence.
We will consider the nested sets  $$(W_j\times \D_r(x_{i_{j-1}})) \cap F_n^{-1} (W_{j+1}\times \D_r(x_{i_j})) \cap \cdots \cap F_n^{-m}( W_{j+m}\times \D_r(x_{i_{j+m-1}})),$$
 and show that the intersection for all $m\in \N$  is a unique  holomorphic disk which is  a holomorphic  graph $$ \D_r(x_{i_j})\ni z\mapsto (\varphi(z),z)\in W_j\times \D_r(x_{i_{j-1}}),$$ and which is actually the local stable manifold of a saddle periodic point.

Define the compact and forward invariant set
$$
\Gamma:= \bigcup_{j=1, \ldots , N} \left(\bigcap_{m \ge 0} F_n^{-m} ( \overline W_{j+m} \times \overline{\D_r(x_{i_{j+m-1}})}) \right).
$$
Let $D$ be the intersection of a  $(i_j, i_{j-1})$-disk  with $W_j\times \D_r(x_{i_{j-1}})$. We know that the image
$F_n(D)$ contains a holomorphic graph  over  the disk $$\D_{R-|\delta| r}(x_{i_{j+1}})\supset\supset\D_{r}(x_{i_{j+1}}).$$ So
 the modulus of the annulus $D \setminus F_n^{-1} (W_{j+1} \times \D_r(x_{i_j}))$ is bounded away from zero. Applying this observation repeatedly and using the Gr\"oztsch Inequality we have that $D \cap \Gamma$ consists of a single point.

 Applying this argument to the trivial foliation of $W_j\times \D_r(x_{i_{j-1}})$ consisting of disks $D$ of the form $\{w=c\}$ we immediately get that $\Gamma \cap (W_j\times \D_r(x_{i_{j-1}}))$ is a graph $ z\mapsto (\varphi(z),z)$ for some function $\varphi\colon  \D_r(x_{i_j})\to W_j$.

We claim that the function $\varphi$ is actually holomorphic.
Recall that in the proof of Lemma \ref{lem:transitive disks} we can choose the ratio between the radii $r$ and $R$ as large as we wish.
The function $f_n$ maps $W_j$ univalently onto $\D_R(x_{i_{j+1}}+\delta x_{i_{j-1}})$.
By applying Cauchy estimates  to $f_n^{-1}$ from $\D_R(x_{i_{j+1}}+\delta x_{i_{j-1}})$ into $\D_r(x_{i_j})$ it follows that $|f_n^\prime(z)|$ can be made arbitrarily large on the subset of $W_j$ that is mapped by $f_n$ onto
$$\D_{r+|\delta |r}(x_{i_{j+1}}+\delta x_{i_{j-1}})\subset \subset\D_R(x_{i_{j+1}}+\delta x_{i_{j-1}}) .$$
 It follows that we may assume that the derivative $|f_n^\prime|$ is arbitrarily large on $(W_j\times \D_r(x_{i_{j-1}})) \cap (F_n^{-1} (W_{j+1}\times \D_r(x_{i_j}) ))$ for every $j$.

Recall that
$$
DF_n(z,w) = \left(
       \begin{array}{cc}
         f_n^\prime(z) & -\delta \\
         1 & 0 \\
       \end{array}
     \right),
$$
hence when $|f_n^\prime(z)|$ is sufficiently large the horizontal cone field $\mathcal{C}_h$ containing the tangent vectors $(v_1,v_2)$ with $|v_2| \le 2|v_1|$ is forward invariant. Let $\mathcal{C}_v$ be the vertical cone field, given by the pullback under $dF_n$ of the constant vertical cone field defined by $|v_2| \ge 2|v_1|$. It follows that $\mathcal{C}_v$ is backwards invariant for any point in $F_n(W_j \times \D_r(x_{i_{j-1}}))$, and moreover, any non-constant tangent vector in $\mathcal{C}_v$ is contracted by some uniform factor, while vectors in $\mathcal{C}_h$ are uniformly expanded. Thus $\Gamma$ is a hyperbolic forward invariant set by the cone criterion, and through every point $(z,w) \in \Gamma$ there exists a stable manifold $W^s(z,w)$. It immediately follows that $\Gamma \cap (W_j\times \D_r(x_{i_{j-1}}))$ has to  coincide with a local stable manifold, and thus the function $\varphi$ is actually holomorphic.

By the forward invariance of $\Gamma$ we know that the holomorphic disk $\Gamma \cap (W_j\times \D_r(x_{i_{j-1}}))$ is mapped into itself by $F_n^N$. The existence of a saddle periodic orbit of period $N$ follows.

Since the maps $F_n$ are all conjugate to $F$ it follows that $F$ has infinitely many periodic cycles of any order $N \ge 3$.
\end{proof}

For polynomial H\'enon maps saddle periodic points form a dense subset of the Julia set $J = J^+ \cap J^-$. While the periodic points constructed above in the not quasi-normal setting are all saddle points, it is unclear to the authors whether there also exist (infinitely many) saddle points of any order $N\ge 3$ in the quasi-normal case.

 \section{Arbitrary Growth of entropy} \label{section:lacunary}

In \cite{Dujardin04}, Dujardin constructed   transcendental H\'enon maps with infinite entropy  by letting $f(z)$ be  an entire function which, on suitable disks $D_i$,  is  well approximated by  polynomials of some degree $d_i\ra\infty$,   to deduce that  the corresponding H\'enon map is H\'enon-like
on the bidiscs $D_i\times D_i$ of the same degree $d_i$. It follows that  the H\'enon map has topological entropy at least $\log d_i\ra\infty$. %  with  $d$ which  can be chosen arbitrarily large.

The rate of the growth of entropy is then given by the relation between $d_i$ and the radii of the disks  $D_i$.

In this section we show that the entropy of \emph{lacunary} power series, i.e. power series with mostly vanishing coefficients, can grow at any prescribed rate. We will first prove the statement for entire functions in one variable:

 \begin{theorem}\label{thm:lacunary}
Let $h(R)$ be a continuous positive increasing  function $h:[0,\infty)\rightarrow [0,\infty)$
with $h(0)=0$ and $\lim_{R\rightarrow \infty} h(R)=\infty$.
Then there exists an entire function $f(z)$ and a sequence  of radii $R_j\nearrow \infty$ so that the topological entropy of
$f$ on $\overline{\D_{R_j}}$ equals $h(R_j)$.
\end{theorem}

\begin{lemma}\label{stability}
Let $P(z):=az^n$ with $a\neq 0$ and $n\geq 2$. Let $r>0$, set $R:=|a|r^n$, and assume that $R/2>r$.
Let $g\colon \overline \D_r\to \C$ be a holomorphic function such that $|g(z)|< R/2^n$ for all $z\in \overline \D_r$.
Then the function defined as $f:=P+g$,
$$f: \D_r\cap f^{-1}(\D_{\frac{R}{2}})\to \D_{\frac{R}{2}}$$ is a polynomial-like map of degree $n$.
\end{lemma}
\begin{proof}
The function $f$ satisfies $f(\partial \D_r)\cap  \overline{\D_{R/2}}=\varnothing$ and by Rouch\'e's Theorem the winding number of the curve $f(\partial \D_r)$ around the origin is $n$. It follows that $f:\D_r\cap f^{-1}(\D_{R/2})\to \D_{R/2}$ is a proper map of degree $n$, and by the maximum principle every connected component of its domain is simply connected. To prove that it is polynomial-like it suffices to show that $\D_r\cap f^{-1}(\D_{R/2})$ is connected. Notice   that $|f|>0$  for $|z|>r/2$, hence  all preimages of $0$ under $f$ are contained in  $\D_{r/2}$, and hence all connected components of $f^{-1}(\D_{R/2})$ have to intersect $\D_{r/2}$. On the other hand, $\D_{r/2}\subset f^{-1}(\D_{R/2})$, hence there is only one connected component of $f^{-1}(\D_{R/2})$ in $\D_R$ as claimed.
\end{proof}

Recall that the entropy of a polynomial-like map of degree $n$ is $\log n$. It follows from the fact that such maps are topologically  conjugate (in fact, hybrid conjugate) to a true polynomial of degree $d$ by Douady-Hubbard Straightening Theorem \cite{DH85} in a neighborhood of their Julia set, or one can prove it directly as for polynomials following for example \cite{lyubich}.

\begin{proof}[Proof of Theorem~\ref{thm:lacunary}]
 We construct $f$ as a lacunary series $\sum_{i=1}^\infty a_iz^{n_i}$  with $(a_i)$ positive real numbers. Define $g_j:=\sum_{i\neq j} a_iz^{n_i}$. By choosing $a_i, r_i,n_i$ appropriately we will  ensure that for each $j$   the monomial  $a_j z^{n_j}=f-g_j$ is the leading term on the circle of radius $r_j$, in the precise way needed to apply Lemma~\ref{stability}.

 %%%%%%%%%%%%555

We will construct the series inductively, along with a sequence of radii $(r_j)$ such that for all integer $j\geq 1$ we have 
\begin{align}\label{zero}
&h(r_j)=\log n_j;\\
\label{una}
&|g_j(z)|\leq \frac{r_j}{2^{n_j}},\quad \forall\, z\in \D_{r_j};\\
 \label{edue}
&a_j r_j^{n_j}>2r_j;\\
\label{etre}
&a_j\leq 2^{-(j+1)j/2}.
\end{align}

By \eqref{etre} the series converges to an entire function $f$. By \eqref{una},\eqref{edue}, and Lemma \ref{stability} we immediately obtain that
the topological entropy of $f$ on $\overline{\D_{r_j}}$ equals $\log n_j$, which by \eqref{zero} is equal to $h(r_j)$.

We start setting $a_1=1/2$, $r_1>2$ such that $h(r_1)=\log(n_1)$ for some  integer $n_1 \geq 2$.
We will choose $a_2, r_2, n_2$ such that
\begin{equation}\label{equattro}
a_2r_1^{n_2}\leq \frac{a_1r_1^{n_1}}{2^{n_1+1}},
\end{equation}
and
\begin{equation}\label{ecinque}
a_1r_2^{n_1}\leq \frac{a_2r_2^{n_2}}{2^{n_2+1}}.
\end{equation}
Consider all possible radii $r_2 > r_1$ for which $h(r_2)$ is  of the form $\log(n_2)$  for some integer $n_2$.
Set $a_2:= a_1r_1^{n_1-n_2}/2^{n_1+1}$, which satisfies \eqref{equattro}.
Substituting in \eqref{ecinque} we obtain
$$\left(\frac{r_2}{r_1}\right)^{n_2-n_1}\geq 2^{n_1+n_2+2},$$
which is satisfied once $r_2$ (and hence $n_2$) is chosen large enough.
Notice that $a_2=1/2 \frac{1}{2^2}$, hence \eqref{etre} is satisfied, and similarly if $r_2$ (and hence $n_2$) is chosen large enough \eqref{edue} is satisfied. Iterating this procedure yields the desired series.
\end{proof}

\begin{corollary}\label{cor:lacunary Henon}
Let $h, f$ be as in Theorem~\ref{thm:lacunary}. Then the topological entropy of $F(z,w)=(f(z)-\delta w,z)$ on $\overline\D_{r_j}\times \overline\D_{r_j}$ equals $h(r_j)$ for all $j$ sufficiently large.
\end{corollary}
\begin{proof}
In the proof of Theorem~\ref{thm:lacunary} we obtained a sequence of  disks $\overline\D_{r_j}$ with $r_j \nearrow \infty$ such that  $|f(z)|> (|\delta| + 1)\cdot r_j$ for $|z| = r_j$ and $j$ sufficiently large, and that $f(z)$ winds $n_j$ times around the origin as $z$ runs around the circle $\partial \D_{r_j}$. It follows from Lemma \ref{polynomial->henon} that the restriction of $F$ to the bidisk $\overline\D_{r_j} \times \overline\D_{r_j}$ is a H\'enon-like map of degree $n_j$, which by Theorem \ref{thm:Henonlike entropy} implies that the topological entropy on $\overline\D_{r_j}\times\overline \D_{r_j}$ equals $h(r_j)$ for all $j$ sufficiently large.
\end{proof}

\bibliographystyle{amsalpha}
\bibliography{Henon2}

\end{document}